\documentclass[11pt]{article}

\usepackage[a4paper,vmargin={25mm,25mm},hmargin={25mm,25mm}]{geometry}
\usepackage{amsmath,amssymb,amsthm,mathtools}
\usepackage{enumitem}
\usepackage{hyperref}
\usepackage{microtype}

\hypersetup{colorlinks=true,linkcolor=blue,citecolor=blue,urlcolor=blue}

\newtheorem{theorem}{Theorem}[section]
\newtheorem{lemma}[theorem]{Lemma}
\newtheorem{proposition}[theorem]{Proposition}

\newcommand{\R}{\mathbb R}
\newcommand{\eps}{\varepsilon}
\newcommand{\diam}{\operatorname{diam}}
\newcommand{\dist}{\operatorname{dist}}
\newcommand{\vol}{\operatorname{vol}}
\newcommand{\cN}{\mathcal N}
\newcommand{\cA}{\mathcal A}

\title{Spectral Bounds for Antipodal Graphs}
\author{Samuel Korsky}
\date{February 28, 2026}

\begin{document}
\maketitle

\begin{abstract}
\noindent
Suppose $\left\{x_1, \dots, x_n\right\} \subset \mathbb{R}^2$ is a set of $n$ points in the plane with diameter $\leq 1$, meaning $|x_i - x_j| \leq 1$ for all $1 \leq i,j \leq n$. We show that the ratio of the number of ``neighbors'' (ordered pairs of points with distance $\leq \varepsilon$) to the number of ``antipodes'' (ordered pairs of points with distance $\geq 1 - \varepsilon$) is $\gtrsim\varepsilon^{1/2 + o(1)}$, attaining the conjectured correct asymptotic within a polylog factor and improving the $\gtrsim\varepsilon^{3/4+o(1)}$ bound of Steinerberger (2025). In dimensions $d\ge3$ we prove a similar result with exponent $3(d - 1)/4$.
\end{abstract}

\section{Introduction}

Let $P=\{x_1,\ldots,x_n\}\subset\R^d$ have diameter at most $1$.  For $0<\eps<1/10$, define the ordered counts
\[
   \cN_\eps(P):=\#\{(i,j): |x_i-x_j|\le \eps\}
\]
and
\[
   \cA_\eps(P):=\#\{(i,j): |x_i-x_j|\ge 1-\eps\}.
\]

\bigskip
\noindent
Steinerberger proved in the plane that
\[
   \cN_\eps(P)
   \gtrsim
   \frac{\eps^{3/4}}{\log(1/\eps)^{1/4}}\cdot\cA_\eps(P)
\]
and observed that the expected optimal exponent is $1/2$ \cite{Steinerberger2025}.  The two standard examples are points distributed on a circle of diameter $1$, and a Reuleaux-type example consisting of many points on a circular arc together with a smaller cluster at the corresponding center.  In both cases the ratio between neighbor pairs and antipodal pairs is of order $\eps^{1/2}$.

\bigskip
\noindent
Our first result reaches this exponent up to a logarithm:

\begin{theorem}[Planar estimate]\label{thm:plane}
There are absolute constants $c>0$ and $\eps_0>0$ such that, for every $0<\eps<\eps_0$ and every finite $P\subset\R^2$ with $\diam(P)\le1$,
\[
   \cN_\eps(P)
   \ge
   c\cdot\frac{\eps^{1/2}}{\log(1/\eps)^{1/2}}\cdot\cA_\eps(P).
\]
\end{theorem}

\noindent
In dimension $d$, the most immediate discretization has about $\eps^{-(d-1)}$ cells near the boundary of the convex hull.  Bounding the spectral radius of the resulting graph by the number of cells gives only the trivial estimate
\[
   \cN_\eps(P)\gtrsim_d \eps^{d-1}\cA_\eps(P).
\]
A common-neighborhood argument improves this to the exponent $d-3/2$.  We also prove a second, trace-based estimate using directional caps which is stronger in dimensions $d\ge4$.

\begin{theorem}[Higher-dimensional estimate]\label{thm:higher}
Let $d\ge3$.  There are constants $c_d>0$ and $\eps_d>0$ such that, for every $0<\eps<\eps_d$ and every finite $P\subset\R^d$ with $\diam(P)\le1$,
\[
   \cN_\eps(P)
   \ge
   c_d\,\eps^{\alpha_d}\,\cA_\eps(P),
   \qquad
   \alpha_d=\frac{3(d-1)}4.
\]
\end{theorem}

\noindent
Note that the conjectural higher-dimensional exponent is $(d-1)/2$, which is far stronger than the bounds proven in this note.

\medskip
\noindent\textbf{Notation.}
We write $A\lesssim_d B$ if $A\le C_dB$ for a constant depending only on $d$, and $A\lesssim B$ if the constant is absolute.  The value of an implicit constant may change from line to line.

\section{Discretization and Spectral Reduction}

We proceed exactly as in \cite{Steinerberger2025}. Let $K=\operatorname{conv}(P)$.  We first restrict to a thin boundary layer:

\begin{lemma}[Interior points are not near-diameter points]\label{lem:interior}
Let $K\subset\R^d$ be convex with $\diam(K)\le1$.  If $x\in K$ and $\dist(x,\partial K)>\eps$, then $|x-y|<1-\eps$ for every $y\in K$.
\end{lemma}

\begin{proof}
Fix $y\ne x$ and set $u=(y-x)/|y-x|$.  Since $\dist(x,\partial K)>\eps$, there is $\eta>0$ such that $x-(\eps+\eta)u\in K$.  Hence
\[
   |y-x|+\eps+\eta
   =|y-(x-(\eps+\eta)u)|
   \le \diam(K)
   \le 1.
\]
Thus $|x-y|<1-\eps$.  The case $x=y$ is trivial.
\end{proof}

\smallskip
\noindent
We shall also use the following standard covering fact for convex bodies.  It follows, for instance, from the Steiner formula and the local surface-area bound for convex hypersurfaces; see, e.g., \cite[Chapters 4 and 5]{Schneider2014}.

\begin{lemma}[Boundary-layer decomposition]\label{lem:cover}
Let $K\subset\R^d$ be convex with $\diam(K)\le1$, and let $0<\eps<1$.  The layer
\[
   K_\eps:=\{x\in K:\dist(x,\partial K)\le\eps\}
\]
can be covered by sets $Q_1,\ldots,Q_k$ of diameter at most $\eps/10$, with
\[
   k\lesssim_d \eps^{-(d-1)}.
\]
Moreover, if $r\ge\eps$, then every ball of radius $r$ intersects at most
\[
   O_d\!\left((r/\eps)^{d-1}\right)
\]
of the sets $Q_a$.  Finally, the cover may be chosen so that, for any set $U\subset\R^d$ and any $O_d(\eps)$-neighborhood $U^+$ of $U$,
\[
   \#\{a:Q_a\cap U\ne\varnothing\}
   \lesssim_d \eps^{-d}\,\vol(U^+\cap K_{O_d(\eps)}).
\]
\end{lemma}

\begin{proof}
Take the nonempty intersections of $K_\eps$ with a grid of half-open cubes of side comparable to $\eps$.  The global and local cardinality bounds follow from the standard tube-volume estimates
\[
   \vol\big((K_\eps)_{C\eps}\big)\lesssim_d \eps,
   \qquad
   \vol\big((K_\eps\cap B(z,r))_{C\eps}\big)\lesssim_d \eps r^{d-1},
\]
valid for convex bodies of diameter at most $1$ and $r\ge\eps$.  The last assertion follows from the same grid construction: every selected cube is contained in an $O_d(\eps)$-neighborhood of $U$.
\end{proof}

\smallskip
\noindent
We now build the graph.  By Lemma \ref{lem:interior}, only points of $P\cap K_\eps$ can occur in $\cA_\eps(P)$.  Cover $K_\eps$ by cells $Q_1,\ldots,Q_k$ as in Lemma \ref{lem:cover}.  Assign each point of $P\cap K_\eps$ to one cell and let
\[
   n_a:=\#(P\cap Q_a)
\]
be the multiplicity of cell $Q_a$.

\bigskip
\noindent
Define the antipodal cell graph $G=(V,E)$ by $V=\{1,\ldots,k\}$ and
\[
   a\sim b
   \quad\Longleftrightarrow\quad
   \sup_{x\in Q_a,\,y\in Q_b}|x-y|\ge1-\eps.
\]
Let $M$ be its adjacency matrix.  For $\eps$ sufficiently small this graph has no loops, since each cell has diameter at most $\eps/10<1-\eps$.  Also, all pairs of points assigned to the same cell are $\eps$-neighbors, and so
\[
   \cN_\eps(P)\ge \sum_{a=1}^k n_a^2=\|n\|_2^2.
\]
Conversely, every near-diameter pair of points is assigned to adjacent cells.  Therefore
\[
   \cA_\eps(P)
   \le
   \sum_{a,b}M_{ab}n_an_b
   =\langle n,Mn\rangle
   \le
   \lambda_1(M)\|n\|_2^2.
\]
We have proved the following reduction.

\begin{proposition}[Spectral reduction]\label{prop:spectral-reduction}
If $\cA_\eps(P)>0$, then
\[
   \frac{\cN_\eps(P)}{\cA_\eps(P)}
   \ge
   \lambda_1(M)^{-1},
\]
where $M$ is the adjacency matrix of the antipodal cell graph.  If $\cA_\eps(P)=0$, the desired inequalities are trivial.
\end{proposition}

\section{The Collatz--Wielandt Estimate}

For $a\in V$, write
\[
   N(a):=\{b\in V:a\sim b\},
   \qquad
   d_a:=|N(a)|.
\]

\begin{lemma}[Collatz--Wielandt degree estimate]\label{lem:CW}
Let $G$ be a finite undirected graph with adjacency matrix $M$.  Then
\[
   \lambda_1(M)
   \le
   \max_{a:d_a>0}\left(\sum_{b\in N(a)}d_b\right)^{1/2}.
\]
\end{lemma}

\smallskip
\begin{proof}
Remove isolated vertices.  The vector $x_a=\sqrt{d_a}$ is then positive.  By the Collatz--Wielandt formula \cite{Collatz1942} and Cauchy's inequality,
\[
   \frac{(Mx)_a}{x_a}
   =
   \frac1{\sqrt{d_a}}\sum_{b\in N(a)}\sqrt{d_b}
   \le
   \left(\sum_{b\in N(a)}d_b\right)^{1/2}.
\]
Taking the maximum over $a$ gives the result.
\end{proof}

\smallskip
\noindent
The weighted degree sum can be written as a common-neighborhood count:
\[
   \sum_{b\in N(a)}d_b
   =
   \sum_{b\in V}|N(a)\cap N(b)|.
\]
This identity is the source of the next two estimates.

\section{Common-Neighborhood Bounds}

\begin{lemma}[Intersection of two thin shells]\label{lem:shell}
Let $d\ge2$ be fixed.  Let $x,y\in\R^d$ and put $r=|x-y|$.  If $A_d\eps\le r\le1$, then the intersection
\[
\begin{aligned}
   &\{z:1-A_d\eps\le |z-x|\le1+A_d\eps\} \\
   &\qquad\cap
   \{z:1-A_d\eps\le |z-y|\le1+A_d\eps\}
\end{aligned}
\]
can be covered by
\[
   O_d\!\left(\frac{\eps^{-(d-2)}}{r}\right)
\]
sets of diameter at most $\eps$.
\end{lemma}

\smallskip
\begin{proof}
After a rigid motion, take $x=(-r/2,0,\ldots,0)$ and $y=(r/2,0,\ldots,0)$.  Subtracting the two squared-distance equations shows that the first coordinate of any point in the intersection varies over an interval of length $O_d(\eps/r)$.  The remaining level set has dimension $d-2$ and uniformly bounded diameter, while the two shell inequalities thicken it by $O_d(\eps)$ in the normal directions.  Thus the intersection is contained in an $O_d(\eps)$-neighborhood of a $(d-2)$-dimensional set of bounded measure, times one additional interval of length $O_d(\eps/r)$.  Covering this set by $\eps$-balls gives $O_d(\eps^{-(d-2)}/r)$ balls.
\end{proof}

\begin{proposition}[Planar weighted degree bound]\label{prop:weighted2}
For the planar antipodal cell graph,
\[
   \sum_{b\in N(a)}d_b
   \lesssim
   \eps^{-1}\log(1/\eps)
\]
for every vertex $a$.
\end{proposition}

\begin{proof}
Fix $a$.  Let $W$ be the set of cells at distance at most $A\eps$ from $Q_a$, where $A$ is a sufficiently large absolute constant.  By Lemma \ref{lem:cover}, $|W|=O(1)$, and the contribution of $W$ to $\sum_b|N(a)\cap N(b)|$ is $O(\eps^{-1})$.

\bigskip
\noindent
Now let $b\notin W$ and put $r=\dist(Q_a,Q_b)$.  If $c\in N(a)\cap N(b)$, then, after enlarging constants to account for the diameters of the cells, $Q_c$ meets the intersection of two shells of thickness $O(\eps)$ around points of $Q_a$ and $Q_b$.  Lemma \ref{lem:shell} gives
\[
   |N(a)\cap N(b)|\lesssim \frac1r.
\]
Consequently, if $|N(a)\cap N(b)|\ge s$, then $r\lesssim s^{-1}$.  Lemma \ref{lem:cover} then implies
\[
   \#\{b\notin W:|N(a)\cap N(b)|\ge s\}
   \lesssim
   \frac1{s\eps}.
\]
Using layer cake summation,
\[
\begin{aligned}
   \sum_{b\notin W}|N(a)\cap N(b)|
   &=
   \sum_{s=1}^{k}\#\{b\notin W:|N(a)\cap N(b)|\ge s\}  \\
   &\lesssim
   \eps^{-1}\sum_{s=1}^{k}\frac1s
   \lesssim
   \eps^{-1}\log(1/\eps),
\end{aligned}
\]
since $k\lesssim\eps^{-1}$ in the plane.
\end{proof}

\begin{proposition}[Higher-dimensional weighted degree bound]\label{prop:weightedD}
Let $d\ge3$.  For the antipodal cell graph in $\R^d$,
\[
   \sum_{b\in N(a)}d_b
   \lesssim_d
   \eps^{-(2d-3)}
\]
for every vertex $a$.
\end{proposition}

\begin{proof}
Fix $a$, and again discard the $O_d(1)$ cells within distance $A_d\eps$ of $Q_a$.  Their contribution is at most $O_d(k)$, with $k\lesssim_d\eps^{-(d-1)}$, which is smaller than the claimed bound.

\bigskip
\noindent
For a remaining cell $b$, put $r=\dist(Q_a,Q_b)$.  Lemma \ref{lem:shell} gives
\[
   |N(a)\cap N(b)|
   \lesssim_d
   \frac{\eps^{-(d-2)}}{r}.
\]
Set $H=\eps^{-(d-2)}$.  If $|N(a)\cap N(b)|\ge s$, then $r\lesssim_d H/s$.  By Lemma \ref{lem:cover},
\[
   \#\{b:|N(a)\cap N(b)|\ge s\}
   \lesssim_d
   k\cdot\min\left\{1,\left(\frac Hs\right)^{d-1}\right\}.
\]
Therefore
\[
\begin{aligned}
   \sum_b|N(a)\cap N(b)|
   &\lesssim_d
   k\cdot\sum_{s=1}^{k}\min\left\{1,\left(\frac Hs\right)^{d-1}\right\} \\
   &\lesssim_d
   k\cdot\left(H+H^{d-1}\sum_{s>H}s^{-(d-1)}\right).
\end{aligned}
\]
Since $d\ge3$, the tail is $O_d(H^{-(d-2)})$.  Thus the last expression is $O_d(kH)$, and
\[
   kH\lesssim_d
   \eps^{-(d-1)}\eps^{-(d-2)}
   =
   \eps^{-(2d-3)}.
\]
\end{proof}

\smallskip
\noindent
Combining Lemma \ref{lem:CW} with Propositions \ref{prop:weighted2} and \ref{prop:weightedD} gives
\[
   \lambda_1(M)
   \lesssim
   \eps^{-1/2}\log(1/\eps)^{1/2}
   \qquad(d=2),
\]
and
\[
   \lambda_1(M)
   \lesssim_d
   \eps^{-(d-3/2)}
   \qquad(d\ge3).
\]

\smallskip
\noindent
For $d \ge 3$ this bound is subsumed by our trace bound in the next section, but was included for completeness.

\section{Directional Trace Bound}

The preceding higher-dimensional estimate comes from two-shell intersections.  We now prove an independent edge-count estimate which is stronger for $d\ge4$.

\bigskip
\noindent
For a convex body $K\subset\R^d$, let
\[
   h_K(u):=\sup_{x\in K}x\cdot u,
   \qquad
   w_K(u):=h_K(u)+h_K(-u),
   \qquad u\in S^{d-1}.
\]
For $t>0$ define the two caps
\[
   C_+(u,t):=\{x\in K:h_K(u)-x\cdot u\le t\},
\]
and
\[
   C_-(u,t):=\{x\in K:h_K(-u)+x\cdot u\le t\}.
\]
Let $N_\pm(u,t)$ denote the number of cells $Q_a$ meeting $C_\pm(u,t)$.

\begin{lemma}[Reciprocal cap count]\label{lem:capcount}
There is a dimensional constant $B_d$ such that, if $w_K(u)\ge1-B_d\eps$, then
\[
   N_+(u,B_d\eps)\,N_-(u,B_d\eps)
   \lesssim_d
   \eps^{-(d-1)}.
\]
\end{lemma}

\begin{proof}
Write $m=d-1$ and let $\pi_u:\R^d\to u^\perp$ be orthogonal projection.  Write $B=B_d$.  If $x\in C_+(u,B\eps)$ and $y\in C_-(u,B\eps)$, then
\[
   (x-y)\cdot u
   \ge
   w_K(u)-2B\eps
   \ge
   1-O_d(\eps),
\]
provided $B_d$ is chosen sufficiently large.  Since $\diam(K)\le1$, this implies
\[
   |\pi_u x-\pi_u y|
   \lesssim_d
   \eps^{1/2}.
\]
Fixing one point in each cap, we get
\[
   \diam\big(\pi_u C_+(u,B\eps)\big)
   +
   \diam\big(\pi_u C_-(u,B\eps)\big)
   \lesssim_d
   \eps^{1/2}.
\]
Thus each projected cap has $m$-dimensional volume $O_d(\eps^{m/2})$.  Each cap has thickness $O_d(\eps)$ in the $u$ direction, and so its $d$-dimensional volume is $O_d(\eps^{m/2+1})$.  Enlarging by the cell scale changes only the implicit constant.  Lemma \ref{lem:cover} therefore gives
\[
   N_\pm(u,B\eps)
   \lesssim_d
   \eps^{-d}\eps^{m/2+1}
   =
   \eps^{-m/2}.
\]
Multiplying the two estimates proves the lemma.
\end{proof}

\smallskip
\noindent
\begin{proposition}[Directional trace estimate]\label{prop:trace}
For the antipodal cell graph in $\R^d$,
\[
   \lambda_1(M)
   \lesssim_d
   \eps^{-3(d-1)/4}.
\]
\end{proposition}

\begin{proof}
Let $E_{\rm ord}:=\{(a,b):a\sim b\}$ be the set of ordered edges.  Since $M$ is a $0$-$1$ symmetric matrix,
\[
   \lambda_1(M)^2\le\operatorname{tr}(M^2)=|E_{\rm ord}|.
\]

\smallskip
\noindent
For each ordered edge $(a,b)$, choose $x\in Q_a$ and $y\in Q_b$ with $|x-y|\ge1-\eps$, and put $v=(x-y)/|x-y|$.  If $u\in S^{d-1}$ satisfies $|u-v|\le c\eps^{1/2}$ for a small dimensional constant $c>0$, then
\[
   (x-y)\cdot u\ge1-O_d(\eps).
\]
Since $(x-y)\cdot u\le w_K(u)\le\diam(K)\le1$, it follows that $w_K(u)\ge1-O_d(\eps)$.  Moreover, if $z\in K$ satisfies $z\cdot u=h_K(u)$, then $z\cdot u-y\cdot u\le |z-y|\le1$, and hence
\[
   h_K(u)-x\cdot u\le 1-(x-y)\cdot u=O_d(\eps).
\]
Thus $x\in C_+(u,B_d\eps)$ after increasing $B_d$; similarly $y\in C_-(u,B_d\eps)$.  Therefore $Q_a$ meets $C_+(u,B_d\eps)$ and $Q_b$ meets $C_-(u,B_d\eps)$.

\bigskip
\noindent
Hence every ordered edge is counted for a set of directions of spherical measure $\gtrsim_d\eps^{(d-1)/2}$.  By Fubini and Lemma \ref{lem:capcount},
\[
\begin{aligned}
   |E_{\rm ord}|\,\eps^{(d-1)/2}
   &\lesssim_d
   \int_{\{u:w_K(u)\ge1-B_d\eps\}}
   N_+(u,B_d\eps)N_-(u,B_d\eps)\,d\sigma(u)  \\
   &\lesssim_d
   \eps^{-(d-1)}.
\end{aligned}
\]
Thus
\[
   |E_{\rm ord}|\lesssim_d \eps^{-3(d-1)/2},
\]
and so
\[
   \lambda_1(M)\lesssim_d \eps^{-3(d-1)/4}.
\]
\end{proof}

\section{Proofs of the Main Theorems}

\begin{proof}[Proof of Theorem \ref{thm:plane}]
The planar common-neighborhood estimate gives
\[
   \lambda_1(M)
   \lesssim
   \eps^{-1/2}\log(1/\eps)^{1/2}.
\]
The spectral reduction, Proposition \ref{prop:spectral-reduction}, gives
\[
   \cN_\eps(P)
   \gtrsim
   \frac{\eps^{1/2}}{\log(1/\eps)^{1/2}}\cdot\cA_\eps(P).
\]
\end{proof}

\begin{proof}[Proof of Theorem \ref{thm:higher}]
For $d\ge3$, Proposition \ref{prop:trace} gives
\[
   \lambda_1(M)\lesssim_d \eps^{-3(d-1)/4}.
\]
Applying Proposition \ref{prop:spectral-reduction} completes the proof.
\end{proof}

\end{document}